\documentclass[a4paper,reqno,index]{amsart}
\reversemarginpar
\usepackage{amsmath, amsfonts, amsthm, amssymb}
\usepackage[maxbibnames=9]{biblatex}
\usepackage{mathrsfs}
\usepackage{fancyhdr}
 
\usepackage{enumerate}   
\usepackage{blkarray}
\usepackage{mathrsfs}
\usepackage[title]{appendix}
\usepackage{amsthm}
\usepackage{xcolor}
\usepackage{graphicx}
\usepackage[normalem]{ulem}
\usepackage{morefloats}
\usepackage{float}

\usepackage{hyperref}
\usepackage{mathtools}

\usepackage[normalem]{ulem} 

 \numberwithin{equation}{section}

\allowdisplaybreaks

\makeindex

\hypersetup{
    bookmarks=true,         
    unicode=false,          
    pdftoolbar=true,        
    pdfmenubar=true,        
    pdffitwindow=false,     
    pdfstartview={FitH},    
    pdftitle={My title},    
    pdfauthor={Author},     
    pdfsubject={Subject},   
    pdfcreator={Creator},   
    pdfproducer={Producer}, 
    pdfkeywords={keyword1, key2, key3}, 
    pdfnewwindow=true,      
    colorlinks=false,       
    linkcolor=green,          
    citecolor=green,        
    filecolor=green,      
    urlcolor=green,           
    urlbordercolor={1 1 1}  
}

\newtheorem{theorem}{Theorem}[section]
\newtheorem*{theorem*}{Theorem}
\newtheorem*{conjecture*}{Conjecture}
\newtheorem{lemma}[theorem]{Lemma}

\renewcommand{\epsilon}{\varepsilon}
\renewcommand{\phi}{\varphi}
\renewcommand{\kappa}{\varkappa}

\theoremstyle{definition}

\newtheorem{question}[theorem]{Question}

\theoremstyle{remark}
\newtheorem{remark}[theorem]{Remark}

 \usepackage{etoolbox}

\theoremstyle{definition}

\theoremstyle{remark}

\begin{document}
\title[\footnotesize Spectral Estimates over Multiply Connected Domains]{\large Spectral Estimates over Multiply Connected Domains}

\author[\footnotesize ]{Georgios Tsikalas}
\address{Department of Mathematics, UC Santa Barbara, Santa Barbara, CA}
\email{gtsikalas@ucsb.edu}

\thanks{}

\small
\begin{abstract}
    \small
       We consider the quantum analogue of a disk with $n$ pairwise disjoint circular holes. We establish a uniform upper bound for the associated spectral constant that is independent of the geometry of the holes, together with a sharper asymptotic bound when the holes are well-separated. The latter generalizes recent estimates for the quantum annulus due to Crouzeix and Pascoe.

\end{abstract}
\maketitle

\section{Introduction}

Let $D(c, r)$ denote the open disk with center $c$ and radius $r.$ Given $R>0$ and disks $D(c_1, r_1), \dots, D(c_n, r_n)$ with pairwise disjoint closures and such that $|c_i|+r_i<R$ for all $i,$ define 
\[D_R\big((c_1, r_1), \dots, (c_n, r_n)\big):=D(0, R)\setminus \bigcup_{1\le i\le n} \overline{D(c_i, r_i)},\]
which is an open disk with $n$ circular holes. The  \textit{Badea-Beckermann-Crouzeix class} is defined  as
\begin{align*}
& \mathscr{D}_R\big((c_1, r_1), \dots, (c_n, r_n)\big) \\ 
=\ &\big\{A\in\mathbb{C}^{d\times d} : ||A||< R \ \text{ and }\ ||(A-c_i)^{-1}||<r^{-1}_i, \ 1\le i \le n. \big\},
\end{align*}
while the associated spectral constant is defined as the smallest constant \[K= K_R\big((c_1, r_1), \dots, (c_n, r_n)\big)\] such that 
\[||f(A)||\le K\sup_{z\in D_R((c_1, r_1), \dots, (c_n, r_n))}|f(z)| \]
whenever $A\in \mathscr{D}_R\big((c_1, r_1), \dots, (c_n, r_n)\big)$ and $f$ is holomorphic in $D_R\big((c_1, r_1), \dots, (c_n, r_n)\big)$.
 \par 
The simplest configuration results from taking $n=1, c_1=0$ and $r_1=1/R$ with $R>1$; this yields $D_R\big((0, 1/R)\big)=\{1/R<|z|<R\}$, and $\mathscr{D}_R\big((0, 1/R)\big)$ becomes the \textit{quantum annulus}. The study of $K_R\big((0, 1/R)\big)$ was initiated by Shields \cite{Shieldsbook}, while the first uniform (not depending on $R$) upper bound was found by Badea-Beckermann-Crouzeix \cite{BBCintersections}. The precise value of the constant remains unknown, though results of Crouzeix \cite{crouzeixRecentAnnulus} and of the author \cite{firstpaper} imply that
\[2\le K_R\big((0, 1/R)\big) \le 1+\sqrt{1+\frac{2}{R^2+1}}, \qquad R>1,\]
and thus $\lim_{R\to\infty}K_R\big((0, 1/R)\big)=2.$ 
The asymptotic behavior was, in fact, first determined by Pascoe \cite{Pascoecross} via dilation theory (see also the recent preprint \cite{PascoeTomar} for alternative approaches). \par 
This note is primarily motivated by the above asymptotic estimate. In particular, we are interested in configurations for $D_R\big((c_1, r_1), \dots, (c_n, r_n)\big)$ where the circular holes sit in a region that is small compared to the size of $\{|z|<R\}$ and are also  far from each other. More precisely, let $d_{ij}$ denote the distance from $D(c_i, r_i)$ to $D(c_j, r_j)$,  with  $D(c_0, r_0)=D(0, R)$ and $i\neq j$.
We aim to estimate $K_R\big((c_1, r_1), \dots, (c_n, r_n)\big)$ when
\begin{equation}\label{sep}
   \max_{0\le j\le n}\Big\{\sum_{1\le i\neq j}\cfrac{r_i}{d_{ij}} +\cfrac{R}{d_{0j}}-1\Big\}  <\epsilon 
\end{equation} 
 for some small $\epsilon>0.$

 \begin{theorem} \label{main}
Fix $R>0$, $c_1, \dots, c_n\in\mathbb{C}$ and $r_1, \dots, r_n>0,$ with $|c_i|+r_i<R$, for all $i$,  such that the disks $\overline{D(c_i, r_i)}$ do not overlap. Assume \eqref{sep} is satisfied
for some $\epsilon>0.$ Then, 
\begin{equation}\label{main-ineq}
 K_R\big((c_1, r_1), \dots, (c_n, r_n)\big)\le \cfrac{n+1}{2}+\sqrt{\frac{(n+1)^2}{4}+5\epsilon}. 
\end{equation} 
 \end{theorem}
 \noindent Thus, as $\epsilon\to 0$, we obtain the asymptotic upper bound of $n+1$ for the spectral constant. This recovers $\lim_{R\to\infty}K_R\big((0, 1/R)\big)=2$ in the case of the annulus. The double-layer potential kernel appears as an essential tool in the proof of Theorem \ref{main}; our arguments are heavily inspired by ideas and techniques of Malman-Mashreghi-O'Loughlin-Ransford \cite{MalmanBartosz} and Schwenninger-de Vries \cite{SchwenningerdeVries}. Indeed, the proof of Theorem \ref{main} is essentially based on the observation that, if \eqref{sep} is satisfied for some small $\epsilon>0$, then the Cauchy transform of any conjugate analytic function over 
 $D_R\big((c_1, r_1), \dots, (c_n, r_n)\big)$ will be close, in the supremum norm, to an appropriately chosen constant. 
 \par 
 Finally, we also provide a uniform upper bound. The main result of \cite{BBCintersections} implies that 
$K_R\big((c_1, r_1), \dots, (c_n, r_n)\big)$ is bounded above by $n+n(n-1)/\sqrt{3}$ (see also the recent paper \cite{Woerdemanintersections}). A slight modification of the proof of Theorem \ref{main} yields the following partial improvement:
\[K_R\big((c_1, r_1), \dots, (c_n, r_n)\big)\le \cfrac{n+1}{2}+\sqrt{\frac{(n+1)^2}{4}+n}, \]
for any configuration of interior, non-overlapping disks. This is Theorem \ref{main2}. 
\par 
Section \ref{prelims} contains basic preliminaries on the double-layer potential. Theorems \ref{main} and \ref{main2} are proved in Section \ref{mainsresults}, which also contains an open question.

\section{Main Results}

\subsection{Preliminaries}\label{prelims}

Set $\Omega=D_R\big((c_1, r_1), \dots, (c_n, r_n)\big)$,  $C_0=\partial D(0, R)$ and $C_i=\partial D(c_i, r_i),$ for $1\le i \le n$. Let $\mathcal{A}(\Omega)$ denote the set of all functions $f$ that are holomorphic in the interior and extend continuously to the boundary of $\Omega$. Further, let $\sigma(s),$ $s\in [0, L]$, denote the arc-length parametrization of $\partial \Omega$, with counterclockwise orientation for the exterior circle and clockwise orientation for the $n$ inner circles. In particular, we write $[0, L]=[0, L_0]\cup[L_0, L_1]\cup\cdots\cup[L_{n-1}, L_n],$ with $[L_{i-1}, L_{i}]$ parametrizing $C_i$ (we set $L_{-1}=0$).

For $z\neq \sigma(s)$, the double-layer potential kernel $\mu$ is defined by 
\[\mu(\sigma(s), z):=\cfrac{1}{\pi}\cfrac{d\ \arg (\sigma(s)-z)}{ds}=\cfrac{1}{2\pi i}\bigg(\cfrac{\sigma'(s)}{\sigma(s)-z}-\cfrac{\overline{\sigma'(s)}{}}{\overline{\sigma(s)}-\overline{z}} \bigg).\]

\noindent  Further, given $f\in\mathcal{A}(\Omega)$, $z\in \Omega$ and a matrix $A$ with spectrum in $\Omega,$ we have the usual Cauchy formulas:

\[f(z)=\cfrac{1}{2\pi i}\int_{\partial \Omega}\cfrac{f(\sigma)}{\sigma- z}\ d\sigma, \ \ \ f(A)=\cfrac{1}{2\pi i}\int_{\partial\Omega}f(\sigma)(\sigma I-A)^{-1}\ d\sigma\]
and 
\[g(z):=C(\overline{f}, z):=\cfrac{1}{2\pi i}\int_{\partial \Omega}\cfrac{\overline{f(\sigma)}}{\sigma- z}\ d\sigma,  \ \ \ g(A)=\cfrac{1}{2\pi i}\int_{\partial\Omega}\overline{f(\sigma)}(\sigma I-A)^{-1}\ d\sigma. \]

\noindent  We also define the  integral transforms with respect to $\mu$:
\[S(f, z):=\int_{[0, L]}f(\sigma(s))\mu(\sigma(s), z)\ ds, \ \ \ S(f, A):=\int_{[0, L]}f(\sigma(s))\mu(\sigma(s), A)\ ds,\]
where 
\[\mu(\sigma(s), A)=\cfrac{1}{2\pi i}\big( \sigma'(s)(\sigma(s)I-A)^{-1}- \overline{\sigma'(s)}(\overline{\sigma(s)}I-A^*)^{-1}\big).\]

If $A\in  \mathscr{D}_R((c_1, r_1), \dots, (c_n, r_n))$, we have the following lower bounds for $\mu(\sigma, A).$
 
\begin{lemma} \label{muA-inf} 
Assume $A\in \mathscr{D}_R((c_1, r_1), \dots, (c_n, r_n))$. Then,

\[\mu(\sigma, A)\ge \begin{cases}
    \cfrac{1}{2\pi R}I, \ \ \ \ &\forall \sigma\in C_0 \\
    -\cfrac{1}{2\pi r_i}I, \ \ \ \ &\forall  \sigma\in C_i,\ i\neq 0. 
\end{cases}\]
\end{lemma}
\begin{proof}
The inequalities follow immediately from \cite[Lemma 6]{Crouzeix-Greenbaum} and \cite[Lemma 7]{Crouzeix-Greenbaum}.
\end{proof}
We also require the following result on extremal functions.
\begin{lemma}[\cite{crouzeixmultiplyPick}]\label{extremal}
Let $A$ be a matrix with eigenvalues in $\Omega.$ There exists $f_0\in\mathcal{A}(\Omega)$ that realizes
\[||f_0(A)||=\max\big\{||f(A)|| : f\in\textup{Hol}(\Omega),\ \sup_{z\in\Omega}|f(z)|\le 1 \big\} \]
\end{lemma}
\noindent In fact, $f_0$ is a uniquely defined inner function.

\subsection{Main Results} \label{mainsresults}
\begin{proof}[Proof of Theorem \ref{main}] Fix $\Omega=D_R\big((c_1, r_1), \dots, (c_n, r_n)\big)$ and $A\in \mathscr{D}_R\big((c_1, r_1), \dots, (c_n, r_n)\big)$.  Assume $\epsilon>0$ is such that \eqref{sep} holds. 
Let $f_0$ be an extremal function as in Lemma \ref{extremal}. If $||f_0(A)||=1,$ there is nothing to prove, so assume $||f_0(A)||>1$. Consider a unit vector $x_0$ such that $||f_0(A)x_0||=||f_0(A)||.$ By \cite{Holbrooksmall}, we have $\langle f_0(A)x_0, x_0\rangle=0$. Set $g_0=C(\overline{f_0})$ and
\begin{align}
c_1&=\cfrac{1}{2\pi R}\int_{[0, L_0]}f_0(\sigma(s))\ ds-\sum_{i=1}^n\cfrac{1}{2\pi r_i}\int_{[L_{i-1}, L_i]}f_0(\sigma(s))\ ds\notag  \\
    c_2&=\cfrac{1}{2\pi R}\int_{[0, L_0]}\overline{f_0(\sigma(s))}\ ds.\label{c2}
\end{align}
Since $f_0(A)=S(f_0, A)-g_0(A)^*$, we have
\begin{align}
||f_0(A)x_0||^2&=\langle f_0(A)x_0, f_0(A)x_0\rangle \notag \\
 &=\langle f_0(A)x_0, S(f_0, A)x_0\rangle -\langle f_0(A)x_0, g_0(A)^*x_0\rangle \notag \\ 
&=\langle f_0(A)x_0, [S(f_0, A)-c_1I]x_0\rangle -\langle f_0(A)[g_0(A)-c_2I]x_0, x_0\rangle \label{decomp}
\end{align}
Now, set 
\begin{equation}\label{nu}
 \nu(\sigma,A)=\begin{cases} \mu(\sigma, A)-\cfrac{1}{2\pi R}I, & \sigma\in C_0 \\ 
\mu(\sigma, A)+\cfrac{1}{2\pi r_i}I & \sigma\in C_i, \ i\neq 0.
\end{cases} 
\end{equation}
By Lemma \ref{muA-inf}, $\nu(\sigma, A)\ge 0$ for all $\sigma\in \partial\Omega.$ Also, $\sigma(A)\subset\Omega$ implies $\int_{[0, L]}\mu(\sigma(s),A)\ ds=2$.
Thus, we may write 
\begin{align}
||S(f_0, A)-c_1I||=& \bigg|\bigg|\int_{[0, L]}\nu(\sigma(s), A)f_0(\sigma(s))\ ds\bigg|\bigg| \notag\\ 
\le &\bigg|\bigg| \int_{[0, L]}\nu(\sigma(s), A)\ ds\bigg|\bigg|\notag \\
= \ &\bigg|\bigg|\int_{[0, L]}\mu(\sigma(s), A)\ d s-\int_{[0, L_0]}\cfrac{1}{2\pi R}\ ds+\sum_{i=1}^n\int_{[L_{i-1}, L_i]}\cfrac{1}{2\pi r_i}\ d s\bigg|\bigg|\notag \\
= \ & 2-1+n \notag\\
= \ & n+1. \label{c1-est}
\end{align}
Next, by  \cite[Theorem 4.5]{Manycrouzeix}, there exists a probability measure $\mu$ such that 
\begin{equation} \label{int-rep}
    \langle h(A)x_0, x_0 \rangle=\int_{\partial\Omega}h(\sigma)\ d\mu(\sigma), \qquad h\in\mathcal{A}(\Omega).
\end{equation} 
 Also, by basic potential theory, we know that, for any $\tau\in\partial\Omega,$
 \[g_0(\tau)=\int_{[0, L]}\mu(\sigma(s),\tau)\overline{f_0(\sigma(s))}\ ds.\]
We will now estimate $|g_0-c_2|.$  Recall that $d_{ij}$  denotes the distance from $D(c_i, r_i)$ to $D(c_j, r_j)$. We will use the elementary estimate
\begin{equation*}
 |\mu(\sigma, z)|\le \cfrac{1}{\pi d_{ij}} \qquad \sigma\in C_i,\ z\in C_j,\ i\neq j,  
\end{equation*}
repeatedly. First, assume $\tau\in C_0.$ Then, $\mu(\sigma, \tau)=1/(2\pi R)$ whenever $\sigma\in C_0,$ and thus
\begin{align}
&|g_0(\tau)-c_2|\notag \\ 
=\ &\bigg|\int_{[0, L_0]}\mu(\sigma(s),\tau)\overline{f_0(\sigma(s))}\ ds+\sum_{i=1}^n\int_{[L_{i-1}, L_i]}\mu(\sigma(s), \tau)\overline{f_0(\sigma(s))}\ ds-c_2 \bigg|  \notag\\
=\ &\bigg|\cfrac{1}{2\pi R}\int_{[0, L_0]}\overline{f_0(\sigma(s))}\ ds+\sum_{i=1}^n\int_{[L_{i-1}, L_i]}\mu(\sigma(s), \tau)\overline{f_0(\sigma(s))}\ ds-c_2 \bigg| \notag \\
=\ &\bigg|\sum_{i=1}^n\int_{[L_{i-1}, L_i]}\mu(\sigma(s), \tau)\overline{f_0(\sigma(s))}\ ds \bigg| \notag \\
\le\ & 2\sum_{i=1}^n\cfrac{r_i}{R-|c_i|-r_i} \notag \\
=\ & 2\sum_{i=1}^n\cfrac{r_i}{d_{0i}} \notag \\
\le\ & 2\epsilon.
\label{est1}
\end{align}
Now, let $\tau\in C_j, j\neq 0.$ First, observe that, since $\int_{\partial\Omega}\frac{f(\sigma)}{\sigma-c_j}\ d\sigma=0 $, 
\begin{align}
& -\cfrac{1}{2\pi r_j}\int_{[L_{j-1}, L_j]}\overline{f_0(\sigma(s))}\ ds+c_2 \notag\\
=\ &\cfrac{1}{2\pi i}\int_{C_j}\cfrac{\overline{f_0(\sigma)}}{\sigma-c_j}\ d\sigma+\cfrac{1}{2\pi i}\int_{C_0}\cfrac{\overline{f_0(\sigma)}}{\sigma}\ d\sigma \notag\\ 
=\ &-\cfrac{1}{2\pi i}\sum_{0<i\neq j}\int_{C_i}\cfrac{\overline{f_0(\sigma)}}{\sigma-c_j}\ d\sigma
-\cfrac{1}{2\pi i}\int_{C_0}\cfrac{\overline{f_0(\sigma)}}{\sigma-c_j}\ d\sigma +\cfrac{1}{2\pi i}\int_{C_0}\cfrac{\overline{f_0(\sigma)}}{\sigma}\ d\sigma \notag\\ 
=\ &-\cfrac{1}{2\pi i}\sum_{0<i\neq j}\int_{C_i}\cfrac{\overline{f_0(\sigma)}}{\sigma-c_j}\ d\sigma
-\cfrac{c_j}{2\pi i}\int_{C_0}\cfrac{\overline{f_0(\sigma)}}{\sigma(\sigma-c_j)}\ d\sigma \notag,
\end{align}
and thus
\begin{align}\label{eq1}
\bigg|  -\cfrac{1}{2\pi r_j}\int_{[L_{j-1}, L_j]}\overline{f_0(\sigma(s))}\ ds+c_2\bigg| & \le \sum_{0<i\neq j}\cfrac{r_i}{d_{ij}}+\cfrac{|c_j|}{R-|c_j|} \notag \\
& \le \sum_{0<i\neq j}\cfrac{r_i}{d_{ij}}+\cfrac{R}{d_{0j}}-1 \notag \\
& \le \epsilon.
\end{align}
Further, recall that, if $\sigma\in C_0$ and $z\in\overline{\Omega},$ then $\big(2\pi \mu(\sigma,z )\big)^{-1}$ equals the radius of the unique circle passing through $z$ that is contained in $\overline{\Omega}$ and is tangent to $C_0$ at $\sigma$. Consequently, $\tau\in C_j$ yields
\[\cfrac{1}{\pi(R+|c_j|+r_j)} \le \mu(\sigma, \tau) \le \cfrac{1}{\pi(R-|c_j|-r_j)},    \qquad \sigma\in C_0,\]
and thus
\begin{equation*} 
    -\cfrac{|c_j|+r_j}{\pi R(R+|c_j|+r_j)} \le \mu(\sigma, \tau)-\cfrac{1}{\pi R} \le \cfrac{|c_j|+r_j}{\pi R(R-|c_j|-r_j)},    \qquad \sigma\in C_0.
\end{equation*}
 This, in particular, yields
\begin{align}
\bigg|\int_{[0, L_0]}\mu(\sigma(s), \tau)\overline{f_0(\sigma(s))}\ ds -2c_2\bigg|& =\bigg|\int_{[0, L_0]}\bigg(\mu(\sigma(s), \tau)-\cfrac{1}{\pi R}\bigg)\overline{f_0(\sigma(s))}\ ds\bigg|\notag \\
& \le  
2 \cfrac{|c_j|+r_j}{R-|c_j|-r_j} \notag 
\\
& =
2 (-1+R/d_{0j}) \notag 
\\
& \le 2\epsilon. \label{eq2}
\end{align}
Combining \eqref{eq1}-\eqref{eq2}, we now obtain
\begin{align}
&|g_0(\tau)-c_2| \notag \\
\le\ & \bigg| \int_{[0, L_0]}\mu(\sigma(s), \tau)\overline{f_0(\sigma(s))}\ ds-2c_2 \bigg|  \notag \\ 
&+\bigg| \sum_{1\le i\neq j}\int_{[L_{i-1}, L_{i}]}\mu(\sigma(s), \tau) \overline{f_0(\sigma(s))}\ ds\bigg|  +  \bigg|\int_{[L_{j-1}, L_j]}\mu(\sigma(s), \tau)\overline{f_0(\sigma(s))}\ ds+c_2\bigg| \notag  \\
\le\ & \bigg| \int_{[0, L_0]}\mu(\sigma(s), \tau)\overline{f_0(\sigma(s))}\ ds-2c_2 \bigg|  \notag \\ 
&+2\sum_{1\le i\neq j}\cfrac{r_i}{d_{ij}} + \bigg| -\cfrac{1}{2\pi r_j}\int_{[L_{j-1}, L_j]}\overline{f_0(\sigma(s))}\ ds+c_2 \bigg|\notag  \\
\le\ & 5\epsilon.\label{est2}
\end{align}
\eqref{est1} and \eqref{est2} yield $||g_0-c_2||_{\infty}<5\epsilon.$ Combining with \eqref{decomp}-\eqref{int-rep}, we obtain
\begin{align}
||f_0(A)||^2& \le ||f_0(A)||\cdot ||S(f_0,A)-c_1I||+\bigg|\int_{\partial\Omega}f_0(\sigma)(g_0(\sigma)-c_2)\ d\mu(\sigma)\bigg| \notag  \\
&\le (n+1)||f_0(A)||+||g_0-c_2||_{\infty} \notag \\ 
&\le (n+1)||f_0(A)||+5\epsilon. \label{almost-final}
\end{align}
\noindent From this last inequality, we finally  obtain
\begin{align*}
||f_0(A)||\le \ &  \cfrac{1}{2}(n+1)+\sqrt{\cfrac{1}{4}(n+1)^2+5\epsilon},
\end{align*}
as desired. 

\end{proof}
 We now prove our uniform estimate for $K_R\big((c_1, r_1), \dots, (c_n, r_n)\big)$, which is restated here for the reader's convenience.
\begin{theorem}\label{main2}
Fix $R>0$, $c_1, \dots, c_n\in\mathbb{C}$ and $r_1, \dots, r_n>0,$ with $|c_i|+r_i<R$, for all $i$, and such that the disks $\overline{D(c_i, r_i)}$ do not overlap.  Then,
\[K_R\big((c_1, r_1), \dots, (c_n, r_n)\big)\le \cfrac{n+1}{2}+\sqrt{\frac{(n+1)^2}{4}+n}. \]
\end{theorem}

\begin{proof}
This is, for the most part, a repetition of the previous proof; the only difference is that, instead of defining $c_2$ as in \eqref{c2}, we instead set \[ c_2=\cfrac{1}{2\pi R}\int_{[0, L_0]}\overline{f_0(\sigma(s))}\ ds-\sum_{i=1}^n\cfrac{1}{2\pi r_i}\int_{[L_{i-1}, L_i]}\overline{f_0(\sigma(s))}\ ds. \]
Now, given $\tau\in\partial\Omega,$ define $\nu(\sigma, \tau)$ as in \eqref{nu}. Then,  $\nu(\sigma, \tau)\ge 0$ for all $\tau\in\partial\Omega,$ and, since $\int_{[0, L]}\mu(\sigma, \tau)\ ds=1$ for such $\tau$, we may argue
 as in \eqref{c1-est} to obtain
\begin{align}
|g_0(\tau)-c_2|&=\bigg|\int_{[0, L]}\nu(\sigma, \tau)\overline{f_0(\sigma(s))}\ ds \bigg| \notag \\ 
&\le \int_{[0, L]} \nu(\sigma, \tau)  \ ds \notag \\ 
&= n. \notag      
\end{align}
Finally, arguing as in \eqref{almost-final} gives
\[||f_0(A)||^2\le (n+1)||f_0(A)||+n, \]
from where the desired estimate follows. 
\end{proof}
\begin{remark}
A comparison of \eqref{sep} with the proof of Theorem \ref{main} indicates that the hypothesis is not optimal and can be weakened. We do not pursue this refinement here, as our focus is on asymptotic behavior.
\end{remark}

We conclude with an open question. Note that, in the case of the quantum annulus, the asymptotic upper bound of $2$ is also a uniform lower bound for $K_R\big((0, 1/R)\big)$. Does an analogous lower bound hold for $K_R\big((c_1, r_1), \dots, (c_n, r_n)\big)$?
\begin{question}
Is it true that 
\[K_R\big((c_1, r_1), \dots, (c_n, r_n)\big)\ge n+1\]
for any configuration of interior, non-overlapping disks?
\end{question}

\printbibliography

\end{document}